\documentclass[11pt]{article}

\usepackage[T1]{fontenc}
\usepackage[utf8]{inputenc}
\usepackage{lmodern}
\usepackage{microtype}
\usepackage{geometry}
\usepackage{amsmath,amsthm,amssymb,mathtools}
\usepackage{hyperref}

\hypersetup{
  colorlinks=true,
  linkcolor=blue,
  citecolor=blue,
  urlcolor=blue
}

\newtheorem{theorem}{Theorem}
\newtheorem{proposition}[theorem]{Proposition}
\newtheorem{lemma}[theorem]{Lemma}
\newtheorem{corollary}[theorem]{Corollary}
\theoremstyle{definition}

\newcommand{\Ann}{\operatorname{Ann}}
\newcommand{\Max}{\operatorname{Max}}
\newcommand{\TotQ}[1]{Q(#1)}

\title{An Integrally Closed Reduced Ring with McCoy Localizations That Is Neither McCoy nor Locally a Domain}
\author{HAOTIAN MA\\Zhejiang University}
\date{}

\begin{document}

\maketitle

\begin{abstract}
We construct an explicit commutative ring $R$ that is reduced and integrally
closed, such that $R_{\mathfrak p}$ is an integrally closed McCoy ring for
every maximal ideal $\mathfrak p$ of $R$, while $R$ itself is not a McCoy ring
and is not locally a domain. This gives an affirmative answer to Problem~9 in
\emph{Open Problems in Commutative Ring Theory}. The construction combines
Akiba's Nagata-type example, which already yields an integrally closed reduced
ring with integrally closed domain localizations and a finitely generated ideal
of zero-divisors with zero annihilator, with an explicit local integrally
closed McCoy ring that is not a domain. Taking the direct product of these two
rings preserves the required local McCoy property while retaining the global
failure of the McCoy condition. As a consequence, $R[X]$ is integrally closed
by Huckaba's criterion.
The proof presented in this note was completed by Rethlas \cite{Rethlas2604},
a natural-language automated reasoning system; the author was
responsible for reviewing and checking the argument.
\end{abstract}

\section{Introduction}

A commutative ring $R$ is called a \emph{McCoy ring} if every finitely
generated ideal $I\subseteq Z(R)$ has a nonzero annihilator. In his 1980
paper, Akiba proved that if $R$ is an integrally closed reduced McCoy ring,
then the polynomial ring $R[X]$ is integrally closed \cite[Theorem~3.2]{Akiba}.
He also proved that if $R_M$ is an integrally closed domain for every maximal
ideal $M$ of $R$, then $R[X]$ is integrally closed
\cite[Corollary~1.3]{Akiba}. Huckaba observed that these results imply that
if $R$ is reduced and $R_M$ is an integrally closed McCoy ring for every
maximal ideal $M$, then $R[X]$ is integrally closed \cite[p.~103]{Huckaba}.

Problem~9 in \emph{Open Problems in Commutative Ring Theory} asks whether
there exists an integrally closed reduced ring $R$ such that every maximal
localization $R_M$ is an integrally closed McCoy ring, but $R$ itself is not
a McCoy ring and is not locally a domain. The purpose of this note is to give
an explicit construction of such a ring.

The argument has three ingredients.
First, we use Akiba's Nagata-type example, which already provides an
integrally closed reduced ring $A$ with localizations $A_M$ all integrally
closed domains, but with a finitely generated ideal inside $Z(A)$ having zero
annihilator. Second, we construct an elementary local McCoy ring $B$ which is
integrally closed but not a domain. Finally, we take the direct product
$R=A\times B$ and verify that this simultaneously preserves the local McCoy
property and the failure of the global McCoy property.

\section{The Akiba Factor}

We begin with the factor supplied by Akiba's example.

\begin{proposition}\label{prop:akiba}
Let $k$ be a field. Let $P$ be a set of representatives of the irreducible
polynomials of $k[X,Y]$ modulo associates. For each $f\in P$, set
\[
T_f:=k[X,Y]/(f),
\]
let $T_f^\prime$ be the derived normal ring of $T_f$ in the sense of Akiba,
and let $T_{f,n}^\prime$ be a copy of $T_f^\prime$ for each $n\in\mathbb N$.
Define
\[
S_A:=\prod_{(f,n)\in P\times \mathbb N} T_{f,n}^\prime,
\qquad
S_{A,0}:=\bigoplus_{(f,n)\in P\times \mathbb N} T_{f,n}^\prime,
\qquad
R_{A,0}:=k+S_{A,0}\subseteq S_A.
\]
Let $x_f,y_f$ denote the residue classes of $X,Y$ in $T_f$ and their images
in every copy $T_{f,n}^\prime$. Define $u,v\in S_A$ by
\[
u(f,n)=x_f,
\qquad
v(f,n)=y_f
\qquad
\text{for all }(f,n)\in P\times\mathbb N,
\]
and set
\[
A:=R_{A,0}[u,v].
\]
Let
\[
I_A:=(u,v)\subseteq A.
\]
Then:
\begin{enumerate}
\item $A$ is reduced and integrally closed.
\item For every maximal ideal $M$ of $A$, the localization $A_M$ is an
integrally closed domain.
\item $I_A$ is a finitely generated ideal contained in $Z(A)$ and
$\Ann_A(I_A)=0$.
\item In particular, $A$ is not a McCoy ring.
\end{enumerate}
\end{proposition}

\begin{proof}
Akiba's Nagata example is exactly the ring $A=R_{A,0}[u,v]$ above
\cite[Example (Nagata)]{Akiba}. Akiba proves that $A$ is quasi-normal, that
is, integrally closed in its total quotient ring, and that for the ideal
$I_A=(u,v)$ every nonzero element of $I_A$ is a zero-divisor and
\[
\Ann_A(I_A)=0.
\]
Therefore $A$ is integrally closed, $I_A\subseteq Z(A)$, and $I_A$ is
generated by $u$ and $v$. This proves part~(3).

Akiba also states in the same example that this ring is reduced and that
every localization $A_M$ at a maximal ideal is an integrally closed domain.
Hence parts~(1) and~(2) hold.

Finally, part~(3) exhibits a finitely generated ideal contained in $Z(A)$ with
zero annihilator, so $A$ fails the defining McCoy condition. This proves~(4).
\end{proof}

\section{A Local McCoy Factor}

We now construct a local integrally closed McCoy ring which is not a domain.

\begin{lemma}\label{lem:localfactor}
Let $k$ be a field, and for each $i\ge 1$ let
\[
S_i:=k[[t]],
\qquad
\mathfrak m_i:=t\,k[[t]].
\]
Set
\[
\mathfrak m_B:=\bigoplus_{i\ge 1}\mathfrak m_i
\subseteq
\prod_{i\ge 1} S_i,
\qquad
B:=k+\mathfrak m_B.
\]
Then:
\begin{enumerate}
\item $B$ is a reduced local ring with maximal ideal $\mathfrak m_B$.
\item Every element of $\mathfrak m_B$ is a zero-divisor. In particular,
$Z(B)=\mathfrak m_B$.
\item Every non-zero-divisor of $B$ is a unit. Consequently $\TotQ{B}=B$, so
$B$ is integrally closed.
\item $B$ is a McCoy ring.
\item $B$ is not a domain.
\end{enumerate}
\end{lemma}

\begin{proof}
Since $B$ is a subring of the product $\prod_{i\ge 1} S_i$ of domains, the
ring $B$ is reduced.

Every element of $B$ has the form $c+e$ with $c\in k$ and
$e=(e_i)_i\in \mathfrak m_B$, where $e_i=0$ for all but finitely many $i$.
Modulo $\mathfrak m_B$, only the scalar part survives, so
\[
B/\mathfrak m_B\cong k.
\]
Thus $\mathfrak m_B$ is maximal.

Let $c+e\in B$ with $c\neq 0$. Write
\[
F:=\{i\in \mathbb N : e_i\neq 0\},
\]
which is finite. For each $i\in F$, the element
\[
d_i:=(c+e_i)^{-1}-c^{-1}
\]
lies in $\mathfrak m_i$, because $e_i\in \mathfrak m_i$ and $c$ is a unit in
the discrete valuation ring $S_i$. Let $d=(d_i)_i$, extended by $0$ outside
$F$. Then $d\in \mathfrak m_B$, and one checks inside the product ring that
\[
(c+e)^{-1}=c^{-1}+d\in k+\mathfrak m_B=B.
\]
Hence every element of $B\setminus \mathfrak m_B$ is a unit, and therefore
$B$ is local with maximal ideal $\mathfrak m_B$. This proves~(1).

Let $x=(x_i)_i\in \mathfrak m_B$. Its support is finite, so choose
$j\in\mathbb N$ outside that support and choose $0\neq y\in \mathfrak m_j$.
Let $e^{(j)}\in \mathfrak m_B$ be the element whose $j$-th coordinate is $y$
and whose other coordinates are $0$. Then $e^{(j)}\neq 0$ and
\[
xe^{(j)}=0.
\]
Thus every element of $\mathfrak m_B$ is a zero-divisor. Since units are
never zero-divisors, $Z(B)=\mathfrak m_B$. This proves~(2).

By part~(2), every non-zero-divisor of $B$ lies outside $\mathfrak m_B$, hence
is a unit by part~(1). Therefore localizing at all non-zero-divisors does not
change the ring:
\[
\TotQ{B}=B.
\]
Thus $B$ is integrally closed in its total quotient ring. This proves~(3).

Let $J=(z_1,\dots,z_r)$ be a finitely generated ideal of $B$ contained in
$Z(B)=\mathfrak m_B$. Each generator $z_\nu$ has finite support, so the union
of all supports is finite. Choose $j$ outside that finite union and choose
$0\neq y\in \mathfrak m_j$. Let $e^{(j)}\in \mathfrak m_B$ be defined as
above. Then
\[
e^{(j)}z_\nu=0
\qquad
(\nu=1,\dots,r).
\]
Hence $0\neq e^{(j)}\in \Ann_B(J)$. Therefore $B$ is a McCoy ring. This
proves~(4).

Finally, $\mathfrak m_B\neq 0$, and every nonzero element of $\mathfrak m_B$
is a zero-divisor by part~(2). Hence $B$ is not a domain. This proves~(5).
\end{proof}

\section{Direct Products}

The final argument uses only standard properties of direct products, but we
record them for convenience.

\begin{lemma}\label{lem:product}
Let $A$ and $B$ be commutative rings.
\begin{enumerate}
\item One has
\[
\TotQ{A\times B}\cong \TotQ{A}\times \TotQ{B}.
\]
Consequently, if $A$ and $B$ are reduced and integrally closed, then
$A\times B$ is reduced and integrally closed.
\item The maximal ideals of $A\times B$ are exactly the ideals of the form
$M\times B$ with $M\in \Max(A)$ and the ideals of the form $A\times N$ with
$N\in \Max(B)$. Moreover,
\[
(A\times B)_{M\times B}\cong A_M,
\qquad
(A\times B)_{A\times N}\cong B_N.
\]
\item Suppose $I=(a_1,\dots,a_r)$ is a finitely generated ideal of $A$ such
that $I\subseteq Z(A)$ and $\Ann_A(I)=0$. Then
\[
J:=I\times B
\]
is a finitely generated ideal of $A\times B$ contained in $Z(A\times B)$ and
satisfies $\Ann_{A\times B}(J)=0$.
\end{enumerate}
\end{lemma}

\begin{proof}
An element $(a,b)\in A\times B$ is a non-zero-divisor if and only if $a$ is a
non-zero-divisor of $A$ and $b$ is a non-zero-divisor of $B$. Therefore
inverting all non-zero-divisors yields
\[
\TotQ{A\times B}\cong \TotQ{A}\times \TotQ{B}.
\]
If $A$ and $B$ are reduced, then $A\times B$ is reduced. If
$(x,y)\in \TotQ{A}\times \TotQ{B}$ is integral over $A\times B$, then
projecting any monic integral equation to the two coordinates shows that $x$
is integral over $A$ and $y$ is integral over $B$. Hence $x\in A$ and
$y\in B$, so $A\times B$ is integrally closed. This proves~(1).

The description of $\Max(A\times B)$ is standard. Fix $M\in \Max(A)$. In the
localization $(A\times B)_{M\times B}$, the idempotent $(1,0)$ becomes a
unit because $(1,0)\notin M\times B$. Since
\[
(1,0)(0,1)=0,
\]
it follows that $(0,1)=0$ in the localization. Hence every fraction is equal
to one with second coordinate zero, and the map
\[
\frac{(a,b)}{(s,t)}\longmapsto \frac{a}{s}
\]
defines an isomorphism
\[
(A\times B)_{M\times B}\cong A_M.
\]
The proof of $(A\times B)_{A\times N}\cong B_N$ is symmetric. This proves~(2).

For part~(3), note that $J=I\times B$ is generated by
\[
(a_1,0),\dots,(a_r,0),(0,1),
\]
so $J$ is finitely generated. Let $(a,b)\in J$. Then $a\in I\subseteq Z(A)$,
so there exists $0\neq x\in A$ with $xa=0$. Therefore
\[
(x,0)(a,b)=(0,0),
\]
and $(x,0)\neq (0,0)$. Thus every element of $J$ is a zero-divisor of
$A\times B$, so $J\subseteq Z(A\times B)$.

Now let $(r,s)\in A\times B$ annihilate $J$. Since $(0,1)\in J$, we get
\[
(r,s)(0,1)=(0,s)=(0,0),
\]
hence $s=0$. Also $(r,0)$ annihilates every $(a,0)$ with $a\in I$, so
$r\in \Ann_A(I)=0$. Therefore
\[
\Ann_{A\times B}(J)=0.
\]
This proves~(3).
\end{proof}

\section{The Main Construction}

We can now answer Problem~9.

\begin{theorem}\label{thm:main}
There exists an integrally closed reduced ring $R$ such that:
\begin{enumerate}
\item for every maximal ideal $\mathfrak p$ of $R$, the localization
$R_{\mathfrak p}$ is an integrally closed McCoy ring;
\item $R$ is not a McCoy ring;
\item $R$ is not locally a domain.
\end{enumerate}
In particular, Problem~9 from \emph{Open Problems in Commutative Ring Theory}
has an affirmative answer.
\end{theorem}

\begin{proof}
Let $A$ be the ring of Proposition~\ref{prop:akiba}, let $B$ be the ring of
Lemma~\ref{lem:localfactor}, and define
\[
R:=A\times B.
\]

By Proposition~\ref{prop:akiba}, the ring $A$ is reduced and integrally
closed, every localization $A_M$ at a maximal ideal is an integrally closed
domain, and the ideal
\[
I_A=(u,v)\subseteq A
\]
satisfies
\[
I_A\subseteq Z(A),
\qquad
\Ann_A(I_A)=0.
\]
By Lemma~\ref{lem:localfactor}, the ring $B$ is reduced, integrally closed,
McCoy, and not a domain. Applying Lemma~\ref{lem:product}(1), we conclude that
$R=A\times B$ is reduced and integrally closed.

Next we verify the local condition. Let $\mathfrak p\in \Max(R)$. By
Lemma~\ref{lem:product}(2), either
\[
\mathfrak p=M\times B
\qquad
\text{for some }M\in \Max(A),
\]
or
\[
\mathfrak p=A\times \mathfrak m_B,
\]
where $\mathfrak m_B$ is the maximal ideal of $B$.

In the first case,
\[
R_{\mathfrak p}\cong A_M
\]
by Lemma~\ref{lem:product}(2). Proposition~\ref{prop:akiba} shows that $A_M$
is an integrally closed domain, and every domain is McCoy. Hence
$R_{\mathfrak p}$ is an integrally closed McCoy ring.

In the second case,
\[
R_{\mathfrak p}\cong B
\]
by Lemma~\ref{lem:product}(2), and Lemma~\ref{lem:localfactor} shows that $B$
is an integrally closed McCoy ring. Therefore every maximal localization of
$R$ is an integrally closed McCoy ring.

The ring $R$ is not locally a domain, because at the maximal ideal
\[
\mathfrak p_0:=A\times \mathfrak m_B
\]
we have
\[
R_{\mathfrak p_0}\cong B,
\]
and $B$ is not a domain by Lemma~\ref{lem:localfactor}.

Finally, $R$ is not a McCoy ring. Indeed, Lemma~\ref{lem:product}(3) applied
to the ideal $I_A\subseteq A$ shows that
\[
J:=I_A\times B
\]
is a finitely generated ideal of $R$ satisfying
\[
J\subseteq Z(R)
\qquad
\text{and}
\qquad
\Ann_R(J)=0.
\]
Hence $R$ fails the defining McCoy condition.

Thus the explicit ring
\[
R=
\Bigl(R_{A,0}[u,v]\Bigr)\times
\Bigl(k+\bigoplus_{i\ge 1} t\,k[[t]]\Bigr)
\]
has all the required properties.
\end{proof}

\begin{corollary}
There exists an integrally closed reduced ring $R$ such that $R[X]$ is
integrally closed, every maximal localization $R_{\mathfrak p}$ is an
integrally closed McCoy ring, but $R$ itself is not a McCoy ring and is not
locally a domain.
\end{corollary}

\begin{proof}
Apply Theorem~\ref{thm:main} and Huckaba's observation
\cite[p.~103]{Huckaba}.
\end{proof}

\section{Use of AI}

The proof in this note was obtained with the assistance of Rethlas
\cite{Rethlas2604}, a natural-language automated reasoning system for
automated conjecture resolution. In the present work, Rethlas was used as a
natural-language mathematical reasoning tool for searching for a
counterexample, organizing the construction, and producing a candidate proof.

For this problem, Rethlas supplied the product-ring strategy, the use of
Akiba's Nagata-type example, and the auxiliary local integrally closed McCoy
ring that is not a domain. The proof in its entirety was then verified by the
human author.

\end{document}